\documentclass[11pt, DIV10,a4paper]{article}
\usepackage{color}
\usepackage{float}
\usepackage[bf,small]{caption2}
\usepackage{comment}
\usepackage{amsmath}
\usepackage{bigints}
\usepackage{rotating, booktabs}
\usepackage{amsopn}
\usepackage{amsfonts}
\usepackage{amsthm}
\usepackage{amssymb}
\usepackage{amsbsy}
\usepackage{multirow}
\usepackage{slashbox}
\usepackage{natbib}
\usepackage{tocbibind}
\usepackage[a4paper,colorlinks,breaklinks,bookmarksopen,bookmarksnumbered]{hyperref}
\usepackage[refpage]{nomencl}
\usepackage{makeidx}
\usepackage{pst-all}
\usepackage{epsfig,psfrag}
\usepackage{graphicx}
\usepackage{amsmath}
\usepackage{epstopdf}
\usepackage{tabularx}
\usepackage{subfigure}
\usepackage{setspace}
\usepackage[mathscr]{euscript}
\usepackage[margin=1in]{geometry}
\usepackage{xr-hyper}
\usepackage{amsfonts} 
\usepackage{calc}

\newcommand{\ueq}[1][]{%
  \if\relax\detokenize{#1}\relax
    \sbox0{$\underbrace{=}_{}$}%
    \mathrel{\mathmakebox[\wd0]{=}}
  \else
    \mathrel{\underbrace{=}_{\mathclap{#1}}}
  \fi}

\newcommand {\ctn}{\cite}

\usepackage{url}
\newcommand{\bx}{\mathbf x}
\newcommand{\bX}{\mathbf X}
\newcommand{\by}{\mathbf y}
\newcommand{\bY}{\mathbf Y}

\newtheorem{theorem}{Theorem}

\numberwithin{equation}{section}
\numberwithin{algo}{section}
\numberwithin{table}{section}
\numberwithin{figure}{section}

\newtheorem{lemma}{Lemma}[section]

\newtheorem{as}{Assumption}

\normalsize

\begin{document}

\title{\textbf{Posterior Consistency of Bayesian Inverse Regression and Inverse Reference Distributions}}
\author{Debashis Chatterjee$^{\dag}$ and Sourabh Bhattacharya$^{\dag, +}$ }
\date{}
\maketitle
\begin{center}
$^{\ddag}$ Indian Statistical Institute\\
$+$ Corresponding author:  \href{mailto: bhsourabh@gmail.com}{bhsourabh@gmail.com}
\end{center}

\begin{abstract}
We consider Bayesian inference in inverse regression problems where the objective is to infer about unobserved covariates from observed responses and covariates. 
We establish posterior consistency of such unobserved covariates in Bayesian inverse regression problems under appropriate priors in a leave-one-out
cross-validation setup. We relate this to posterior consistency of inverse reference distributions (\ctn{Bhattacharya13}) for assessing model adequacy.  
We illustrate our theory and methods with various examples of Bayesian inverse regression, along with adequate simulation experiments.
\\[2mm]
\end{abstract}

\section{Introduction}
\label{sec:intro}

Assessment of model adequacy is always fundamental in statistics -- this basic realization has given rise to a huge literature on testing goodness of model fit.
However, compared to the classical literature, the Bayesian literature on model adequacy test is much scarce. 
A comprehensive overview of the existing approaches is provided in \ctn{Vehtari12}. Two relatively prominent existing formal and general 
approaches in this direction are those of \ctn{Gelman96} and \ctn{Bayarri00}. The former relies on posterior predictive $P$-value associated with a discrepancy measure
that is a function of the data as well as the parameters. The latter criticize this approach on account of `double use of the data' and come up with two alternative
$P$-values, demonstrating their advantages over the posterior predictive $P$-value. Indeed, double use of the data prevents the posterior predictive $P$-value to have
uniform distribution on $[0,1]$, while the $P$-values of \ctn{Bayarri00} at least asymptotically has the desired uniform distribution on $[0,1]$.

\ctn{Bhattacharya13} introduced a different approach to Bayesian model assessment in `inverse problems', where the model is built with response
variables and covariates, but unlike `forward problems', the interest is to predict unobserved covariates using the rest of the data, not response variables
from the covariates and the remaining data. The palaeoclimate reconstruction problem provided the necessary motivation, where 
`modern data' consisting of multivariate counts of species and observed climate values, and fossil assemblage data on the same species, deposited in lake sediments
over thousands of years, are available. The interest is to reconstruct the past climate values corresponding to the fossil assemblages using the available data.
Here, the species composition is modeled as a function of climate, since variations in climate is responsible for variations in species composition, but not vice versa.
The inverse nature of the problem is evident since the interest lies in prediction of the past climate variables, not the species composition.
Since the past climates are the unobserved (unknown) covariates, it is natural to consider a prior distribution for such unknown quantities.

The motivating example arises in quantitative palaeoclimate reconstruction where `modern data' consisting of multivariate counts of species
are available along with the observed climate values. Also available are fossil assemblages of the same species, but deposited in lake sediments 
for past thousands of years. This is the fossil species data. However, the past climates corresponding to the fossil species data are unknown, and it is of interest
to predict the past climates given the modern data and the fossil species data. Roughly, the species composition are regarded as functions
of climate variables, since in general ecological terms, variations in climate drives variations in species, but not vice versa.
However, since the interest lies in prediction of climate variables,  the inverse nature of the problem is clear. The past climates, which must be regarded as random variables,
may also be interpreted as {\it unobserved covariates}. It is thus natural to put a prior probability distribution on the unobserved covariates.

Broadly, the model assessment method of \ctn{Bhattacharya13} is based on the simple idea that the model fits the data if the posterior distribution 
of the random variables corresponding to the covariates capture the observed values of the covariates. Now note that the true (observed) values will not be known in reality,
which is why the training data (the modern data in the palaeoclimate problem, for instance) with observed covariates has been considered by \ctn{Bhattacharya13}.
Assuming that the covariates are unobserved, one can predict these values in terms of the posterior distribution of the random quantities standing for the (assumed) missing
covariates. \ctn{Bhattacharya13} demonstrate that it makes more sense to consider leave-one-out cross-validation (LOO-CV) of the covariates 
particularly when some of the model parameters
are given improper prior. From the traditional statistical perspective, LOO-CV is also a very natural method in model assessment. 
We henceforth concentrate on the LOO-CV approach proposed by \ctn{Bhattacharya13}. Briefly, based on the LOO-CV posteriors of the covariates, some appropriate
`inverse reference distribution' (IRD) is constructed. This IRD can be viewed as a distribution of some appropriate statistic associated with the unobserved covariates.
If the distribution captures the observed statistic associated with the observed covariates, then the model is said to fit the data. Otherwise, the model does not fit the data.
\ctn{Bhattacharya13} provide a Bayesian decision theoretic justification of the key idea and show that the relevant IRD based posterior probability analogue of the
aforementioned $P$-values have the uniform distribution on $[0,1]$.
Furthermore, ample simulation studies and successful applications to several real, palaeoclimate models and data sets reported in \ctn{Bhattacharya13},
\ctn{Bhattacharya06} and \ctn{Sabya13}, vindicate the practicality and usefulness of the IRD approach.

The rest of our paper is structured as follows. The general premise of our inverse regression model, LOO-CV and the IRD approach are described in Section \ref{sec:prelims}.  
General consistency issues of the same are discussed in Section \ref{sec:discussion_consistency_IRD}.
We propose an appropriate prior for $\tilde x_i$ and investigate its properties in Section \ref{sec:prior}, and in Section \ref{sec:consistency_loo_cv}
prove consistency of the LOO-CV posteriors under reasonably mild conditions.
Relating consistency of the LOO-CV posteriors, we prove consistency of the IRD approach in Section \ref{sec:consistency_IRD}.
In Section \ref{sec:discussion_results} we provide a discussion on the issues and applicability of our asymptotic theory in various inverse regression contexts 
and in Section \ref{sec:simstudy}, we illustrate our asymptotic theory with simulation studies. Finally, we make concluding remarks in Section \ref{sec:conclusion}.

\section{Preliminaries and general setup}
\label{sec:prelims}

We consider experiment with $n$ covariate observations $x_1, x_2, \ldots, x_n $ along with responses 
$\{y_{ij} :1\leq i\leq n, 1\leq j\leq m\}$.
In other words, the experiment considered here will allow us to have $m$ samples of responses $\{y_{i1}, y_{i2}, \ldots, y_{im}\}$ against covariate observations 
$x_{i}$, for $i=1, 2,\ldots, n$. Both $x_i$ and $y_{ij}$ are allowed to be multidimensional. 
In this article, we consider large sample scenario where both $m, n\rightarrow \infty$. 

For $i=1,\ldots,n$ and $j=1,\ldots,m$, consider the following general model setup: 
conditionally on $x_i$ and $\theta$,
\begin{eqnarray}\label{model}
\begin{aligned}
	& y_{ij} \sim f_{\theta}\left(x_i\right), 
\end{aligned}
\end{eqnarray}
independently.
In (\ref{model}), $f_\theta$ is a known distribution depending upon 
(a set of) parameters $\theta\in\Theta$, where $\Theta$ is the parameter space, which may be infinite-dimensional. 
For the sake of generality, we shall consider $\theta=(\eta,\xi)$, where $\eta$ is a function of the covariates, which we more explicitly denote as $\eta(x)$, 
where $x\in\mathcal X$, $\mathcal X$ being the space of covariates. The part $\xi$ of $\eta$ will be assumed to consist of other parameters, such as the unknown
error variance.

\subsection{Examples of the above model setup}
\label{subsec:model_examples}
\begin{itemize}
\item[(i)] $y_{ij}\sim Bernoulli(p_i)$, where $p_i=H\left(\eta(x_i)\right)$, where $H$ is some appropriate link function and $\eta$ is some 
function with known or unknown form. For known, suitably parameterized form, the model is parametric. If the form of $\eta$ is unknown, 
one may model it by a Gaussian process, assuming adequate smoothness of the function.
\item[(ii)] $y_{ij}\sim Poisson(\lambda_i)$, where $\lambda_i=H\left(\eta(x_i)\right)$, where $H$ is some appropriate link function and $\eta$ is 
some function with known (parametric) or unknown (nonparametric) form. Again, in case of unknown form of $\eta$, 
the Gaussian process can be used as a suitable model under sufficient smoothness
assumptions.
\item[(iii)] $y_{ij}=\eta(x_i)+\epsilon_{ij}$, where $\eta$ is a parametric or nonparametric function and
$\epsilon_{ij}$ are $iid$ Gaussian errors. In particular, $\eta(x_i)$ may be a linear regression function, that is, $\eta(x_i)=\beta'x_i$, where
$\beta$ is a vector of unknown parameters.
Non-linear forms of $\eta$ are also permitted. 
Also, $\eta$ may be a reasonably smooth function of unknown form, modeled by some appropriate Gaussian process.		
\end{itemize}

\subsection{The Bayesian inverse LOO-CV setup and the IRD approach}
\label{subsec:loo_cv}

In the Bayesian inverse LOO-CV setup, for $i\geq 1$, we successively leave out $x_i$ from the data set, and attempt to predict the same
using the rest of the dataset, in the form of the posterior $\pi(\tilde x_i|\bX_{n,-i},\bY_{nm})$, where $\bY_{nm}=\left\{y_{ij}:i=1,\ldots,n;j=1,\ldots,m\right\}$,
$\bX_n=\{x_i:i=1,\ldots,n\}$ and $\bX_{n,-i}=\bX_n\backslash\{x_i\}$, and $\tilde x_i$ is the random quantity corresponding to the left out $x_i$.

In this article, we are interested in proving that $\pi(\tilde x_i\in U^c_i|\bX_{n,-i},\bY_{nm})\rightarrow 0$ almost surely as $m,n\rightarrow\infty$, where $U_i$
is any neighborhood of $x_i$. Here, for any set $A$, $A^c$ denotes the complement of $A$.

Note that the $i$-th LOO-CV posterior is given by 
\begin{equation}
\pi(\tilde x_i|\bX_{n,-i},\bY_{nm})=\int_{\Theta}\pi(\tilde x_i|\theta,\by_i)d\pi(\theta|\bX_{n,-i},\bY_{nm}). 
\label{eq:loo_cv}
\end{equation}

In the IRD approach, we consider the distribution of any suitable statistic $T(\tilde\bX_n)$, where the distribution of 
$\tilde\bX_n=\left\{\tilde x_1,\ldots,\tilde x_n\right\}$
is induced by the respective LOO-CV posteriors of the form (\ref{eq:loo_cv}). The distribution of $T(\tilde\bX_n)$ is referred to as the IRD
in \ctn{Bhattacharya13}. Now consider the observed statistic $T(\bX_n)$. In a nutshell, if $T(\bX_n)$ falls within the desired $100(1-\alpha)\%$ ($0<\alpha<1$)
of the IRD, then the model is said to fit the data; 
otherwise, the model does not fit the data. Typical examples of $T(\bX_n)$, which turned out to be useful in the palaeoclimate modeling context
are (see \ctn{Sabya13}) are: 
\begin{eqnarray}
T_1(\bX_n) &=& \sum_{i=1}^n\frac{(x_i - E_{\pi}(\tilde x_i))^2}{V_{\pi}(\tilde x_i)}
\label{eq:discrepancy1}\\
T_2(\bX_n) &=& \sum_{i=1}^n\frac{|x_i - E_{\pi}(\tilde x_i)|}{\sqrt{V_{\pi}(\tilde x_i)}}
\label{eq:discrepancy2}\\
T_3(\bX_n) &=& x_i
\label{eq:discrepancy3}
\end{eqnarray}
To obtain $T(\tilde\bX_n)$ corresponding to $T(\bX_n)$ above, we only need to replace $x_i$ with $\tilde x_i$ in (\ref{eq:discrepancy1}) -- (\ref{eq:discrepancy3}). 
In the above, $E_{\pi}$ and $V_{\pi}$ denote the expectation and the variance, respectively, with respect to the LOO-CV posteriors.
The statistic $T_3(\tilde\bX_n)$ is $\tilde x_i$ itself, so that the posterior of $T_3(\tilde\bX_n)$ is nothing but the $i$-th LOO-CV posterior. Such a statistic can
be important when there is particular interest in $x_i$, for instance, if one suspects outlyingness of $x_i$. An example of such an issue is considered in 
\ctn{Bhatta07}.

\section{Discussion regarding consistency of the LOO-CV and the IRD approach}
\label{sec:discussion_consistency_IRD}

The question now arises if the IRD approach is at all consistent. That is, whether by increasing $n$ and $m$, the distribution of $T(\tilde\bx)$ will increasingly
concentrate around $T(\bx)$. A sufficient condition for this to hold is consistency of the $i$-th LOO-CV posterior at $x_i$, for $i\geq 1$.
From (\ref{eq:loo_cv}) it is clear that consistency of $\pi(\theta|\bX_{n,-i},\bY_{nm})$ at the truth $\theta_0$ is required for this purpose, but even if
$\theta$ in $\pi(\tilde x_i|\theta,\by_i)$ is replaced with $\theta_0$, consistency of (\ref{eq:loo_cv}) at $x_i$ does not hold for arbitrary priors on $\tilde x_i$,
and for fixed $m\geq 1$. This has been demonstrated in \ctn{Chatterjee17} with the help of a simple Poisson regression with mean $\theta x_i$, where both $\theta$ and $x_i$
are positive quantities. Special priors on $\tilde x_i$ is needed, along with the setup with $m\rightarrow\infty$, to achieve desired consistency of 
the LOO-CV posterior of $\tilde x_i$ at $x_i$. In Section \ref{sec:prior} we propose such an appropriate prior form and establish some requisite properties
of the prior and $\pi(\tilde x_i|\theta,\by_i)$. With such prior and with conditions that ensure consistency of $\pi(\theta|\bX_{n,-i},\bY_{nm})$ at $\theta_0$, 
we establish consistency of the LOO-CV posteriors in Section \ref{sec:consistency_loo_cv}.

Indeed, in the setups that we consider, for any $m\geq 1$, 
$\pi(\theta|\bX_n,\bY_{nm})$ is consistent at the true value $\theta_0$.	
That is, for any neighbourhood $V$ of $\theta_0$, for given $m\geq 1$, $\pi(\theta\in V|\bX_n,\bY_{nm})\rightarrow 1$ almost surely, as $n\rightarrow\infty$.
Assuming complete separable metric space $\Theta$, this is again equivalent to weak convergence of $\pi(\theta|\bX_n,\bY_{nm})$ to $\delta_{\theta_0}$,
as $n\rightarrow\infty$, for $m\geq 1$, for almost all data sequences (see, for example, \ctn{Ghosh03}, \ctn{Ghosal17}).

In our situations, we assume that the conditions of \ctn{Shalizi09} hold for $m\geq 1$, which would ensure consistency of 
$\pi(\theta|\bX_n,\bY_{nm})$ is consistent at the true value $\theta_0$. The advantages of Shalizi's results include great generality of the model and prior
including dependent setups, and reasonably easy to verify conditions. The results crucially hinge on verification of the asymptotic equipartition property.
In Section \ref{subsec:shalizi_briefing} we provide an overview of the main assumptions and result of Shalizi. The full details of the seven assumptions 
($S1$)--($S7$) of Shalizi are provided in the Appendix. 
In Section \ref{subsec:weak_conv} we show that Shalizi's result leads to weak convergence of the posterior of $\theta$ to the point mass at $\theta_0$, which will play
an useful role in our proof of consistency of the LOO-CV posteriors.

\subsection{A briefing of Shalizi's approach}
\label{subsec:shalizi_briefing}
Let $\bY_n=(Y_1,\ldots,Y_n)^T$, and let $f_{\theta}(\bY_n)$ and $f_{\theta_0}(\bY_n)$ denote the observed and the true likelihoods respectively, under the given value of the parameter $\theta$
and the true parameter $\theta_0$. We assume that $\theta\in\Theta$, where $\Theta$ is the (often infinite-dimensional) parameter space. However, it is not required to  
assume that $\theta_0\in\Theta$, thus allowing misspecification. This is the general situation; however, as already mentioned, we do not consider misspecification
for our purpose.
The key ingredient associated with Shalizi's approach to proving convergence of the posterior distribution of $\theta$ is to show that the 
asymptotic equipartition property holds.
To elucidate, let us consider the following likelihood ratio:
\begin{equation*}
R_n(\theta)=\frac{f_{\theta}(\bY_n)}{f_{\theta_0}(\bY_n)}.
\end{equation*}
Then, to say that for each $\theta\in\Theta$, the generalized or relative asymptotic equipartition property holds, we mean
\begin{equation}
\underset{n\rightarrow\infty}{\lim}~\frac{1}{n}\log R_n(\theta)=-h(\theta),
\label{eq:equipartition}
\end{equation}
almost surely, where
$h(\theta)$ is the KL-divergence rate given by
\begin{equation}
h(\theta)=\underset{n\rightarrow\infty}{\lim}~\frac{1}{n}E_{\theta_0}\left(\log\frac{f_{\theta_0}(\bY_n)}{f_{\theta}(\bY_n)}\right),
\label{eq:S3}
\end{equation}
provided that it exists (possibly being infinite), where $E_{\theta_0}$ denotes expectation with respect to the true model.
Let
\begin{align}
h\left(A\right)&=\underset{\theta\in A}{\mbox{ess~inf}}~h(\theta);\notag\\ 
J(\theta)&=h(\theta)-h(\Theta);\notag\\ 
J(A)&=\underset{\theta\in A}{\mbox{ess~inf}}~J(\theta).\notag 
\end{align}
Thus, $h(A)$ can be roughly interpreted as the minimum KL-divergence between the postulated and the true model over the set $A$. If $h(\Theta)>0$, this indicates
model misspecification. 
For $A\subset\Theta$, $h(A)>h(\Theta)$, so that $J(A)>0$.

As regards the prior, it is required to construct an appropriate sequence of sieves $\mathcal G_n$ such that $\mathcal G_n\rightarrow\Theta$ and $\pi(\mathcal G^c_n)\leq\alpha\exp(-\beta n)$,
for some $\alpha>0$. 

With the above notions, verification of (\ref{eq:equipartition}) along with several other technical conditions ensure that for any $A\subseteq\Theta$ such that $\pi(A)>0$, 
\begin{equation}
\underset{n\rightarrow\infty}{\lim}~\pi(A|\bY_n)=0,
\label{eq:post_conv1}
\end{equation}
almost surely, provided that $h(A)>h(\Theta)$.

\subsection{Weak convergence of Shalizi's result}
\label{subsec:weak_conv}

From (\ref{eq:post_conv1}) it follows that for any $\epsilon>0$, 
\begin{equation}
\underset{n\rightarrow\infty}{\lim}~\pi(\mathbb N^c_{\epsilon}|\bY_n)=0, 
\label{eq:cons1}
\end{equation}
where $\mathbb N_{\epsilon}=\left\{\theta:h(\theta)\leq h\left(\Theta\right)+\epsilon\right\}$.
In our case, we shall not consider misspecification, as we are interested in ensuring posterior consistency. Thus, we have $h\left(\Theta\right)=0$ in our context.
Now observe that $h(\theta)$ given by (\ref{eq:S3}) is not a proper KL-divergence between two distributions. Thus the question arises if 
(\ref{eq:cons1}) suffices for posterior consistency, and hence weak convergence of the posterior to $\delta_{\theta_0}$. Lemma \ref{lemma:lemma1} below settles this question
in the affirmative.
\begin{lemma}
\label{lemma:lemma1}
Given any neighborhood $U$ of $\theta_0$, the set $\mathbb N_{\epsilon}$ is contained in $U$ for sufficiently small $\epsilon$. 
\end{lemma}
\begin{proof}
It is sufficient to prove that $h(\theta)>0$ if and only if $\theta\neq\theta_0$.  
Note that 
$E_{\theta_0}\left(\log\frac{f_{\theta_0}(\bY_n)}{f_{\theta}(\bY_n)}\right)$ is a proper KL-divergence and hence is non-decreasing with $n$ (see \ctn{Erven14}). Hence if 
$\theta\neq\theta_0$, then there exists $\varepsilon>0$ such that $E_{\theta_0}\left(\log\frac{f_{\theta_0}(\bY_n)}{f_{\theta}(\bY_n)}\right)>\varepsilon$ for all $n\geq 1$.
Hence, $h(\theta)$ given by (\ref{eq:S3}) is larger than $\varepsilon$ if $\theta\neq\theta_0$.
Of course, if $h(\theta)>0$, we must have $\theta\neq\theta_0$, since otherwise, $E_{\theta_0}\left(\log\frac{f_{\theta_0}(\bY_n)}{f_{\theta}(\bY_n)}\right)=0$
for all $n$, which would imply $h(\theta)=0$.
This proves the lemma.
\end{proof}

It follows from Lemma \ref{lemma:lemma1} that for any neighborhood $U$ of $\theta_0$, $\pi(U|\bY_n)\rightarrow 1$, almost surely, as $n\rightarrow\infty$.
Thus, $\pi(\cdot|\bY_n)\stackrel{w}{\longrightarrow} \delta_{\theta_0}(\cdot)$, almost surely, as $n\rightarrow\infty$, where $``\stackrel{w}{\longrightarrow}"$
denotes weak convergence.

\section{Prior for $\tilde x_i$}
\label{sec:prior}
We consider the following prior for $\tilde x_i$: given $\theta$,
\begin{equation}
	\tilde x_i\sim Uniform\left(B_{im}(\theta)\right),
\label{eq:prior_x}
\end{equation}
where 
\begin{equation}
	B_{im}(\theta)=\left(\left\{x:\eta(x)\in \left[\bar y_i-\frac{cs_i}{\sqrt{m}},\bar y_i+\frac{cs_i}{\sqrt{m}}\right]\right\}\right),
\label{eq:set1}
\end{equation}
In (\ref{eq:set1}), $\bar y_i=\frac{1}{m}\sum_{j=1}^my_{ij}$ and $s^2_i=\frac{1}{m-1}\sum_{j=1}^m(y_{ij}-\bar y_i)^2$, and $c\geq 1$ is some constant.
We denote this prior by $\pi(\tilde x_i|\eta)$. Lemma \ref{lemma:prior_continuity} shows that the density or any probability 
associated with $\pi(\tilde x_i|\eta)$ is continuous with respect to $\eta$.

\subsection{Illustrations}
\label{subsec:illustrations_prior}
\begin{itemize}
	\item[(i)] $y_{ij}\sim Poisson(\theta x_i)$, where $\theta>0$ and $x_i>0$ for all $i$. Here, under the prior $\pi(\tilde x_i|\theta)$, 
		$\tilde x_i$ has uniform distribution on the set 
		$B_{im}(\theta)=\left\{x>0:\frac{\bar y_i-\frac{cs_i}{\sqrt{m}}}{\theta}\leq x\leq \frac{\bar y_i+\frac{cs_i}{\sqrt{m}}}{\theta}\right\}$.
	\item[(ii)] $y_{ij}\sim Poisson(\lambda_i)$, where $\lambda_i=\lambda(x_i)$, with $\lambda(x)=H(\eta(x))$. Here $H$ is a known, one-to-one, 
		continuously differentiable function and $\eta(\cdot)$ is an unknown function modeled by Gaussian process.
		Here, the prior for $\tilde x_i$ is the uniform distribution on $$B_{im}(\eta)=\left\{x:\eta(x)\in 
		H^{-1}\left\{\left[\bar y_i-\frac{cs_i}{\sqrt{m}},\bar y_i+\frac{cs_i}{\sqrt{m}}\right]\right\}\right\}.$$
	\item[(iii)] $y_{ij}\sim Bernoulli(p_i)$, where $p_i=\lambda(x_i)$, with $\lambda(x)=H(\eta(x))$. Here $H$ is a known, increasing,
		continuously differentiable, cumulative distribution function and $\eta(\cdot)$ is an unknown function modeled by some appropriate Gaussian process.
		Here, the prior for $\tilde x_i$ is the uniform distribution on $B_{im}(\eta)=\left\{x:\eta(x)\in 
		H^{-1}\left\{\left[\bar y_i-\frac{cs_i}{\sqrt{m}},\bar y_i+\frac{cs_i}{\sqrt{m}}\right]\right\}\right\}$.
	\item[(iv)] $y_{ij}=\eta(x_i)+\epsilon_{ij}$, where $\eta(\cdot)$ is an unknown function modeled by some appropriate Gaussian process,
		and $\epsilon_{ij}$ are $iid$ zero-mean Gaussian noise with variance $\sigma^2$. 
		Here, the prior for $\tilde x_i$ is the uniform distribution on $B_{im}(\eta)=\left\{x:\eta(x)\in 
		\left[\bar y_i-\frac{cs_i}{\sqrt{m}},\bar y_i+\frac{cs_i}{\sqrt{m}}\right]\right\}$.
\end{itemize}

\subsection{Some properties of the prior}
\label{subsec:prior_properties}
Our proposed prior for $\tilde x_i$ possesses several useful properties necessary for our asymptotic theory. These are formally provided
in the lemmas below.
\begin{lemma}
\label{lemma:prior_continuity}
	The prior density $\pi(\tilde x_i|\eta)$ or any probability associated with $\pi(\tilde x_i|\eta)$ is continuous with respect to $\eta$.
\end{lemma}
\begin{proof}
Let $\left\{\eta_k:k=1,2,\ldots\right\}$ be a sequence of functions such that $\|\eta_k-\eta\|\rightarrow 0$, as $k\rightarrow\infty$,
	where $\|\cdot\|$ denotes the sup norm. It then follows that for any set $A$, 
	$$\left\{x:\eta_k(x)\in A\right\}\cap B_{im}(\eta)\rightarrow \left\{x:\eta(x)\in A\right\}\cap B_{im}(\eta),~\mbox{as}~k\rightarrow\infty.$$
Hence, as $k\rightarrow\infty$,	
	$$Leb\left(\left\{x:\eta_k(x)\in A\right\}\cap B_{im}(\eta)\right)\rightarrow Leb\left(\left\{x:\eta(x)\in A\right\}\cap B_{im}(\eta)\right),$$
where, for any set $A$, $Leb(A)$ denotes the Lebesgue measure of $A$.	
This proves the lemma.
\end{proof}
If the density of $\by_i$ given $\tilde x_i$ and $\theta$, which we denote by $f(\by_i|\theta,\tilde x_i)$,
is continuous in $\theta$ and $\Theta$ is bounded then it would follow from 
Lemma \ref{lemma:prior_continuity} and the dominated convergence theorem that $\pi(\tilde x_i|\theta,\by_i)$ and its associated probabilities are also continuous
in $\theta$. Below we formally present the result as Lemma \ref{lemma:post_continuity}.

\begin{lemma}
\label{lemma:post_continuity}
	If $f(\by_i|\theta,\tilde x_i)$ is continuous in $\theta$ and $\Theta$ is bounded, then the density $\pi(\tilde x_i|\theta,\by_i)$ 
	or any probability associated with $\pi(\tilde x_i|\theta,\by_i)$ is continuous with respect to $\theta$.
\end{lemma}
However, we usually can not assume a compact parameter space. For example, such compactness assumption is invalid for Gaussian process priors for $\theta$. But in most
situations, continuity of the density of $\pi(\tilde x_i|\theta,\by_i)$ and its associated probabilities with respect to $\theta$ hold even without the compactness
assumption, provided $f(\by_i|\theta,\tilde x_i)$ is continuous in $\theta$. We thus make the following realistic assumption:
\begin{as}
$\pi(\tilde x_i|\theta,\by_i)$ is continuous in $\theta$.
\label{as:as1}
\end{as}
The following result holds due to Assumption \ref{as:as1} and Scheffe's theorem (see, for example, \ctn{Schervish95}).
\begin{lemma}
\label{lemma:postprob_continuity}
	If Assumption \ref{as:as1} holds, then any probability associated with $\pi(\tilde x_i|\theta,\by_i)$ is continuous in $\theta$.
\end{lemma}

\section{Consistency of the LOO-CV posteriors}
\label{sec:consistency_loo_cv}

For consistency of the LOO-CV posteriors given by (\ref{eq:loo_cv}), we first need to ensure weak convergence of $\pi(\theta|\bX_{n,-i},\bY_{nm})$ almost surely 
to $\delta_{\theta_0}$, as $n\rightarrow\infty$, for $m\geq 1$. This holds if and only if $\pi(\theta|\bX_n,\bY_{nm})$ is consistent at $\theta_0$. This can be seen by
noting that the $i$-th factor of $\log R_n(\theta)$, obtained by integrating out $\tilde x_i$, does not play any role 
in by (\ref{eq:equipartition}) and (\ref{eq:S3}), so that these limits remain the same as in the case of $\pi(\theta|\bX_n,\bY_{nm})$. The other conditions of Shalizi
also remain the same for both the posteriors $\pi(\theta|\bX_n,\bY_{nm})$ and $\pi(\theta|\bX_{n,-i},\bY_{nm})$.

Hence, assuming that conditions (S1)--(S7) of Shalizi are verified for $\pi(\theta|\bX_n,\bY_{nm})$, for fixed $m$, it follows that 
$\pi(\cdot|\bX_{n,-i},\bY_{nm})\stackrel{w}{\longrightarrow}\delta_{\theta_0}(\cdot)$, almost surely, as $n\rightarrow\infty$.

For any neighborhood $U_i$ of $x_i$, note that the probability $\pi(\tilde x_i\in U^c_i|\theta,\by_i)$ is continuous in $\theta$ 
due to Lemma \ref{lemma:postprob_continuity}.
Moreover, since it is a probability, it is bounded. Hence, by the Portmanteau theorem, using (\ref{eq:loo_cv}) and consistency of 
$\pi\left(\theta|\bX_{n,-i},\bY_{nm}\right)$ it holds almost surely that
\begin{align}
	\pi(\tilde x_i\in U^c_i|\bX_{n,-i},\bY_{nm})&=\int_{\Theta}\pi(\tilde x_i\in U^c_i|\theta,\by_i)d\pi(\theta|\bX_{n,-i},\bY_{nm})\notag\\
	&\stackrel{a.s.}{\longrightarrow}
	\pi(\tilde x_i\in U^c_i|\theta_0,\by_i),~\mbox{as}~n\rightarrow\infty,~\mbox{for any}~m\geq 1.
\label{eq:loo_cons1}
\end{align}
We formalize this result as the following theorem.
\begin{theorem}
	\label{theorem:loo_cons1}
	Assume conditions (S1)--(S7) of Shalizi. Then for $i\geq 1$, under the prior (\ref{eq:prior_x}) and 
	Assumption \ref{as:as1}, (\ref{eq:loo_cons1}) holds almost surely, for any $m\geq 1$,
	for any neighborhood $U_i$ of $x_i$.
\end{theorem}

Let us now make the following extra assumptions: 
\begin{as}
$f(\by_i|\theta_0,\tilde x_i)$ is continuous in $\tilde x_i$.
\label{as:as2}
\end{as}

\begin{as}
$\eta_0$ is a one-to-one function.
\label{as:as3}
\end{as}

With these assumptions, we have the following result.
\begin{theorem}
\label{theorem:theorem1}
	Under the prior (\ref{eq:prior_x}) and Assumptions \ref{as:as2} and \ref{as:as3}, for any neighborhood $U_i$ of $x_i$, for any $i\geq 1$,
\begin{equation}
	\pi(\tilde x_i\in U^c_i|\theta_0,\by_i)\stackrel{a.s.}{\longrightarrow}0,~\mbox{as}~m\rightarrow\infty.
	\label{eq:conv2}
\end{equation}
\end{theorem}
\begin{proof}
Note that
\begin{align}
\pi(\tilde x_i\in U^c_i|\theta_0,\by_i)&
=\frac{\int_{U^c_i}\pi(\tilde x_i|\theta_0)f(\by_i|\theta_0,\tilde x_i)d\tilde x_i}
	{\int_{U^c_i}\pi(\tilde x_i|\theta_0)f(\by_i|\theta_0,\tilde x_i)d\tilde x_i+\int_{U_i}\pi(\tilde x_i|\theta_0)f(\by_i|\theta_0,\tilde x_i)d\tilde x_i}.
	\label{eq:conv3}
\end{align}
Let us consider $\int_{U^c_i}\pi(\tilde x_i|\theta_0)f(\by_i|\theta_0,\tilde x_i)d\tilde x_i$ of (\ref{eq:conv3}). 
Since the support of $\tilde x_i$ is compact, Assumption \ref{as:as2} ensures that $f(\by_i|\theta_0,\tilde x_i)$ is bounded. 
Hence, 
\begin{align}
\int_{U^c_i}\pi(\tilde x_i|\theta_0)f(\by_i|\theta_0,\tilde x_i)d\tilde x_i\leq
K\int_{U^c_i}\pi(\tilde x_i|\theta_0)d\tilde x_i=K\pi(\tilde x_i\in U^c_i|\theta_0), 
\label{eq:conv4}
\end{align}	
	for some positive constant $K$. Now note that $\pi(\tilde x_i\in U^c_i|\theta_0)=\pi(\tilde x_i\in U^c_i\cap B_{im}(\theta_0)|\theta_0)$, and 
	Assumption \ref{as:as3} ensures that $B_{im}(\theta_0)\rightarrow\{x_i\}$ almost surely, as $m\rightarrow\infty$, for all $i\geq 1$.
	It follows that there exists $m_0\geq 1$ such that $U^c_i\cap B_{im}(\theta_0)=\emptyset$, for $m\geq m_0$. Hence, 
	$\pi(\tilde x_i\in U^c_i\cap B_{im}(\theta_0)|\theta_0)\rightarrow 0$, as
$m\rightarrow\infty$. This implies, in conjunction with (\ref{eq:conv4}) and (\ref{eq:conv3}), that (\ref{eq:conv2}) holds.

\end{proof}
Combining Theorems \ref{theorem:loo_cons1} and \ref{theorem:theorem1} yields the following main result.
\begin{theorem}
\label{theorem:theorem2}
	Assume conditions (S1)--(S7) of Shalizi. Then with the prior (\ref{eq:prior_x}), under further Assumptions \ref{as:as1} -- \ref{as:as3}, for $i\geq 1$,
	\begin{equation}
		\underset{m\rightarrow\infty}{\lim}\underset{n\rightarrow\infty}{\lim}~\pi(\tilde x_i\in U^c_i|\bX_{n,-i},\bY_{nm})=0,~\mbox{almost surely},
	\end{equation}
for any neighborhood $U_i$ of $x_i$.
\end{theorem}

\section{Consistency of the IRD approach}
\label{sec:consistency_IRD}
Due to practical usefulness, we consider consistency of IRD associated with (\ref{eq:discrepancy1}) -- (\ref{eq:discrepancy3}). Among these, the IRD associated with $T_3$
is just the $i$-th LOO-CV posterior, which is consistent by Theorem \ref{theorem:theorem2}. For $T_1$ and $T_2$, we consider slight modification by
dividing the right hand sides of (\ref{eq:discrepancy1}) and (\ref{eq:discrepancy2}) by $n$, and adding some small quantity $\varepsilon>0$ to $V_{\pi}(\tilde x_i)$.
These adjustments are not significant for practical applications, but seems to be necessary for our asymptotic theory.
With these, we provide the consistency result and its for the IRD corresponding to $T_1$; that corresponding to $T_2$ would follow in the same way.

\begin{theorem}
\label{theorem:consistency_T1}
Assume conditions (S1)--(S7) of Shalizi, and the prior (\ref{eq:prior_x}). Also let Assumptions \ref{as:as1} -- \ref{as:as3} hold, for $i\geq 1$,
Define for some $\varepsilon>0$, the following: 
$$T_1(\tilde\bX_n) =\frac{1}{n} \sum_{i=1}^n\frac{(\tilde x_i - E_{\pi}(\tilde x_i))^2}{V_{\pi}(\tilde x_i)+\varepsilon}$$
and $$T_1(\bX_n) =\frac{1}{n} \sum_{i=1}^n\frac{(x_i - E_{\pi}(\tilde x_i))^2}{V_{\pi}(\tilde x_i)+\varepsilon}.$$
Then 
\begin{equation}	
	\left|T_1(\tilde\bX_n)-T_1(\bX_n)\right|\stackrel{P}{\longrightarrow}0,~\mbox{as}~m\rightarrow\infty,~n\rightarrow\infty,~\mbox{almost surely}.
	\label{eq:consistency_T1}
\end{equation}
In the above, $``\stackrel{P}{\longrightarrow}"$ denotes convergence in probability.
\end{theorem}
\begin{proof}
The assumptions of this theorem ensures consistency of the LOO-CV posteriors due to Theorem \ref{theorem:theorem2}. This again is equivalent to
almost sure weak convergence of the $i$-th cross-validation posterior to $\delta_{\{x_i\}}$, for $i\geq 1$. This is again equivalent to convergence
in (cross-validation posterior) distribution of $\tilde x_i$, to the degenerate quantity $x_i$, almost surely. Due to degeneracy, this is again equivalent to convergence
in probability, almost surely.

For notational clarity we denote $\tilde x_i$ by $\tilde x^{nm}_i$, whose LOO-CV posterior is $\pi(\cdot|\bX_{n,-i},\bY_{nm})$.  
Let also $\tilde\bX^{nm}=\{\tilde x^{nm}_1,\ldots,\tilde x^{nm}_n\}$, so that we now denote $T_1(\tilde\bX_n)$ by $T_1(\tilde\bX^{nm})$.
It follows from the above arguments that for $i\geq 1$, 
\begin{equation}
\tilde x^{nm}_i\stackrel{P}{\longrightarrow}\tilde x_i,~\mbox{almost surely},~\mbox{as}~m\rightarrow\infty,~n\rightarrow\infty.
\label{eq:cons_T1_2}
\end{equation}
Now consider $T_1(\tilde\bX^{nm})-T_1(\bX_n)$, which is an average of $n$ terms, the $i$-th term being
\begin{equation}
	z^{nm}_i=\frac{(\tilde x^{nm}_i - E_{\pi}(\tilde x^{nm}_i))^2-(x_i - E_{\pi}(\tilde x^{nm}_i))^2}{V_{\pi}(\tilde x^{nm}_i)+\varepsilon}.
\label{eq:cons_T1_3}
\end{equation}
Due to bounded support of $\tilde x^{nm}_i$ and (\ref{eq:cons_T1_2}), uniform integrability entails $E_{\pi}(\tilde x_i)\rightarrow x_i$
and $V_{\pi}(\tilde x_i)\rightarrow 0$, almost surely. The latter two results ensure, along with (\ref{eq:cons_T1_2}), that for $i\geq 1$,
\begin{equation}
z^{nm}_i\stackrel{P}{\longrightarrow}0,~\mbox{as}~m\rightarrow\infty,~n\rightarrow\infty,~\mbox{almost surely}.
\label{eq:cons_T1_4}
\end{equation}
Now note that if $z^{nm}_i$ were non-random, then $z^{nm}_i\rightarrow 0$, as $m\rightarrow\infty$, $n\rightarrow\infty$, would imply
$\frac{1}{n}\sum_{i=1}^nz^{nm}_i\rightarrow 0$ as $m\rightarrow\infty$, $n\rightarrow\infty$.
Hence, by Theorem 7.15 of \ctn{Schervish95} (page 398), it follows that 
$$T_1(\tilde\bX^{nm})-T_1(\bX_n)\stackrel{P}{\longrightarrow}0,~\mbox{as}~m\rightarrow\infty,~n\rightarrow\infty,~\mbox{almost surely}.$$
In other words, (\ref{eq:consistency_T1})) holds.
\end{proof}

\section{Discussion of the applicability of our asymptotic results in the inverse regression contexts}
\label{sec:discussion_results}
From the development of the asymptotic results it is clear that there are two separate aspects that ensures consistency of the LOO-CV posteriors.
The first is consistency of the posterior of the parameter(s) $\theta$, and then consistency of $\pi(\tilde x_i|\theta,\by_i)$. Once consistency
of the posterior of $\theta$ is ensured, our prior for $\tilde x_i$ then guarantees consistency of the posterior of $\tilde x_i$ at $x_i$.
For verify consistency of the posterior of $\theta$, we referred to the general conditions of Shalizi because of their wide applicability, including dependent
setups, and relatively easy verifiability of the conditions. Indeed, the seven conditions of Shalizi have been verified in the contexts of general stochastic process 
(including Gaussian process) regression (\ctn{Chatterjee1}) with both Gaussian and double exponential errors, 
binary and Poisson regression involving general stochastic process (including Gaussian process) and known link functions (\ctn{Chatterjee2})
Moreover, for finite-dimensional parametric problems, the conditions are much simpler to verify. Thus, the examples provided in Section \ref{subsec:illustrations_prior}
are relevant in this context, and the LOO-CV posteriors, and hence the IRD, are consistent. Furthermore, \ctn{Chandra1} and \ctn{Chandra2} establish the
conditions of Shalizi in an autoregressive regression context, even for the so-called ``large $p$, small $n$" paradigm. In such cases, our asymptotic results
for the LOO-CV posteriors and the IRD, will hold. 

There is one minor point to touch upon regarding our requirement for ensuring consistency. In all the aforementioned works regarding verification of Shalizi's conditions,
$m=1$ was considered. For our asymptotic theory, we first require consistency of $\theta$ as $n\rightarrow\infty$, for fixed $m\geq 1$, and then take the limit as
$m\rightarrow\infty$. This is of course satisfied if consistency
holds for $m=1$, as for more information about $\theta$ brought in for larger values of $m$, consistency automatically continues to hold. 
Indeed, for fixed $m\geq 1$, the limit as $n\rightarrow\infty$ does not depend upon $m$, as the posterior of $\theta$ converges weakly to 
the point mass at $\theta_0$, almost surely. Thus, it is always sufficient to verify consistency of the posterior of $\theta$ for $m=1$.

\section{Simulation studies}
\label{sec:simstudy}
\subsection{Poisson parametric regression}
\label{subsec:poisson_parametric}
Let us first consider the case where $y_{ij}\sim Poisson(\theta x_i)$, as briefed in Section \ref{subsec:illustrations_prior} (i). Here we investigate consistency
of the posterior of $\tilde x_i$. We generate the data by simulating $\theta\sim Uniform(0,2)$, $x_i\sim Uniform(0,2)$, $i=1,\ldots,n$, and then by generating
$y_{ij}\sim Poisson(\theta x_i)$, for $j=1,\ldots,m$ and $i=1,\ldots,n$.
We set $\pi(\theta)=1$; $\theta>0$, for the prior for $\theta$.

Since numerical integration turned out to be unstable, we resort to Gibbs sampling from the posterior, noting that the full conditional
distributions of $\theta$ and $\tilde x_i$ are of the forms
\begin{align}
	[\theta|\tilde x_i,\bX_{n,-i},\bY_{nm}]&\propto\theta^{\sum_{i=1}^n\sum_{j=1}^my_{ij}}\exp\left\{-m\theta\left(\tilde x_i+\sum_{j\neq i}x_j\right)\right\}
	I_{\left[\frac{\max\left\{0,\bar y_i-cs_i/\sqrt{m}\right\}}{\tilde x_i},\frac{\bar y_i+cs_i/\sqrt{m}}{\tilde x_i}\right]}(\theta);\notag\\
	[\tilde x_i|\theta,\bX_{n,-i},\bY_{nm}]&\propto\tilde x^{m\bar y_i}_i\exp\left(-m\theta \tilde x_i\right)
	I_{\left[\frac{\max\left\{0,\bar y_i-cs_i/\sqrt{m}\right\}}{\theta},\frac{\bar y_i+cs_i/\sqrt{m}}{\theta}\right]}(\tilde x_i).\notag
\end{align}
It follows that $[\theta|\tilde x_i,\bX_{n,-i},\bY_{nm}]$ has the gamma distribution with shape parameter $\sum_{i=1}^n\sum_{j=1}^my_{ij}+1$
and rate parameter $m\left(\tilde x_i+\sum_{j\neq i}x_j\right)$, truncated on 
$\left[\frac{\max\left\{0,\bar y_i-cs_i/\sqrt{m}\right\}}{\tilde x_i},\frac{\bar y_i+cs_i/\sqrt{m}}{\tilde x_i}\right]$.
Similarly, $[\tilde x_i|\theta,\bX_{n,-i},\bY_{nm}]$ has the gamma distribution with shape parameter $m\bar y_i+1$ and rate parameter $m\tilde x_i$, truncated on
$\left[\frac{\max\left\{0,\bar y_i-cs_i/\sqrt{m}\right\}}{\theta},\frac{\bar y_i+cs_i/\sqrt{m}}{\theta}\right]$. 

For our investigation, we set $i=1$. That is, without loss of generality, we address consistency of the posterior of $\tilde x_1$ via simulation study.  
As for the choice of $c$, we set $c=20$. This choice ensured
that the full conditional distributions have reasonably large support, for given values of $n$ and $m$. We run our Gibbs sampler for $11000$ iterations, and discard
the first $1000$ iterations as burn-in.

Figure \ref{fig:poisson_parametric} displays the posterior densities of $\tilde x_1$ for different values of $m$ and $n$; 
here, for convenience of presentation, we have set $m=n$.
The vertical line denotes the true value $x_1$.
The diagram vividly depicts that the LOO-CV posterior of $\tilde x_1$ concentrates more and more around $x_1$ as $n$ and $m$ increase. 
\begin{figure}
\begin{center}
\includegraphics[width=9cm,height=9cm]{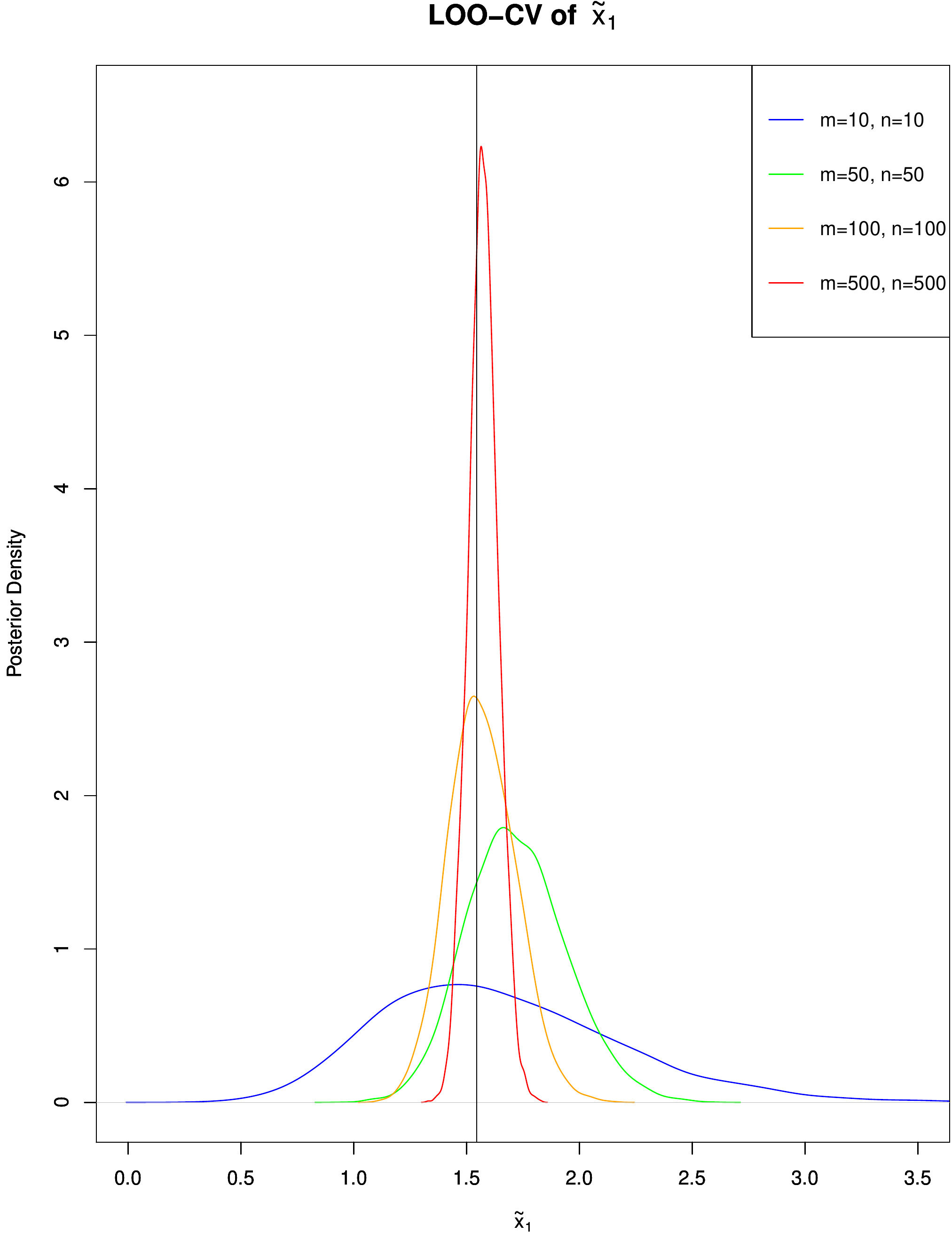}
\end{center}
\caption{Demonstration of posterior consistency in inverse paremetric Poisson regression. The vertical line denotes the true value.}
\label{fig:poisson_parametric}
\end{figure}

\subsection{Poisson nonparametric regression}
\label{subsec:poisson_nonparametric}
We now consider the case where $y_{ij}\sim Poisson(\lambda(x_i))$, where $\lambda(x)=H(\eta(x))$, as briefed in Section \ref{subsec:illustrations_prior} (ii).
In particular, we let $H(\cdot)=\exp(\cdot)$ and $\eta(\cdot)$ be a Gaussian process with mean function $\mu(x)=\alpha+\beta x$ and covariance 
$Cov\left(\eta(x_1),\eta(x_2)\right)=\sigma^2\exp\left\{-(x_1-x_2)^2\right\}$, where $\sigma$ is unknown. We assume that the true data-generating distribution
is $y_{ij}\sim Poisson(\lambda(x_i))$, with $\lambda(x)=\exp(\alpha_0+\beta_0(x))$. We generate the data by simulating $\alpha_0\sim Uniform(-1,1)$, $\beta_0\sim Uniform(-1,1)$
and $x_i\sim Uniform(-1,1)$; $i=1,\ldots,n$, and then finally simulating $y_{ij}\sim Poisson(\lambda(x_i))$; $j=1,\ldots,m$, $i=1,\ldots,n$.

For our convenience, we reparameterize $\sigma^2$ as $\exp(\omega)$, where $-\infty<\omega<\infty$. 
For the prior on the parameters, we set $\pi\left(\alpha,\beta,\omega\right)=1$, for $-\infty<\alpha,\beta,\omega<\infty$. Now note that
the prior for $\tilde x_i$, which is uniform on 
$B_{im}(\eta)=\left\{x:\eta(x)\in H^{-1}\left\{\left[\bar y_i-\frac{c_1s_i}{\sqrt{m}},\bar y_i+\frac{c_2s_i}{\sqrt{m}}\right]\right\}\right\}$, does not have a closed form,
since the form of $\eta(x)$ is unknown. However, if $m$ is large, the interval 
$H^{-1}\left\{\left[\bar y_i-\frac{c_1s_i}{\sqrt{m}},\bar y_i+\frac{c_2s_i}{\sqrt{m}}\right]\right\}$ is small, and $\eta(x)$ falling in this small interval
can be reasonably well-approximated by a straight line. Hence, we set $\eta(x)=\mu(x)=\alpha+\beta x$, for $\eta(x)$ falling in this interval.
In our case, it follows that $[\tilde x_i|\eta]\sim Uniform(a,b)$, where $a=\beta^{-1}\left(\log\left(\bar y_i-\frac{c_1s_i}{\sqrt{m}}\right)-\alpha\right)$
and $b=\beta^{-1}\left(\log\left(\bar y_i+\frac{c_2s_i}{\sqrt{m}}\right)-\alpha\right)$.

We set $c_1=1$ and $c_2=100$, for ensuring positive value of $\bar y_i-\frac{c_1s_i}{\sqrt{m}}$ (so that logarithm of this quantity is well-defined) 
and a reasonably large support of the prior for $\tilde x_i$. As before, we set $i=1$, for our purpose, thus focussing on posterior consistency of $\tilde x_1$ only.

In this example, both numerical integration and Gibbs sampling are infeasible. Hence, we resort to Transformation based Markov Chain Monte Carlo (TMCMC)
(\ctn{Dutta13}) for simulating from the posterior. In particular, we use the additive transformation and update all the unknowns simultaneously, in a single block. 
More specifically, at each iteration $t=1,2,\ldots$, we first generate $\epsilon\sim N(0,1)$, a standard normal variable. Then, letting 
$\left(\tilde x^{(t)}_1,\alpha^{(t)},\beta^{(t)},\omega^{(t)},\eta^{(t)}(x_2),\ldots,\eta^{(t)}(x_n)\right)$
denote the values of the unknowns at the $t$-th iteration, at the $(t+1)$-th iteration we set $\tilde x_1=\tilde x^{(t)}_1\pm 0.5\epsilon$, 
$\alpha=\alpha^{(t)}\pm 0.5\epsilon$, $\beta=\beta^{(t)}\pm 0.5\epsilon$, $\omega=\omega^{(t)}\pm 0.05\epsilon$, and $\eta(x_k)=\eta^{(t)}(x_k)\pm 0.00005\epsilon$; 
$k=2,\ldots,n$. We accept these proposed values with an appropriate acceptance probability (see \ctn{Dutta13} for details), provided the prior conditions are satisfied.  
This strategy has yielded reasonable mixing properties of the additive TMCMC algorithm, for all values of $n$ and $m$ chosen. 
We run our additive TMCMC algorithm for $11000$ iterations, discarding the first $1000$ iterations as burn-in.

Figure \ref{fig:poisson_nonparametric} shows the posterior densities of $\tilde x_1$ for this nonparametric inverse regression problem for different values of $n$ and $m$.
Again, it is clearly evident that the posterior concentrates more and more around the true value $x_1$, as $n$ and $m$ are increased. 
\begin{figure}
\begin{center}
\includegraphics[width=9cm,height=9cm]{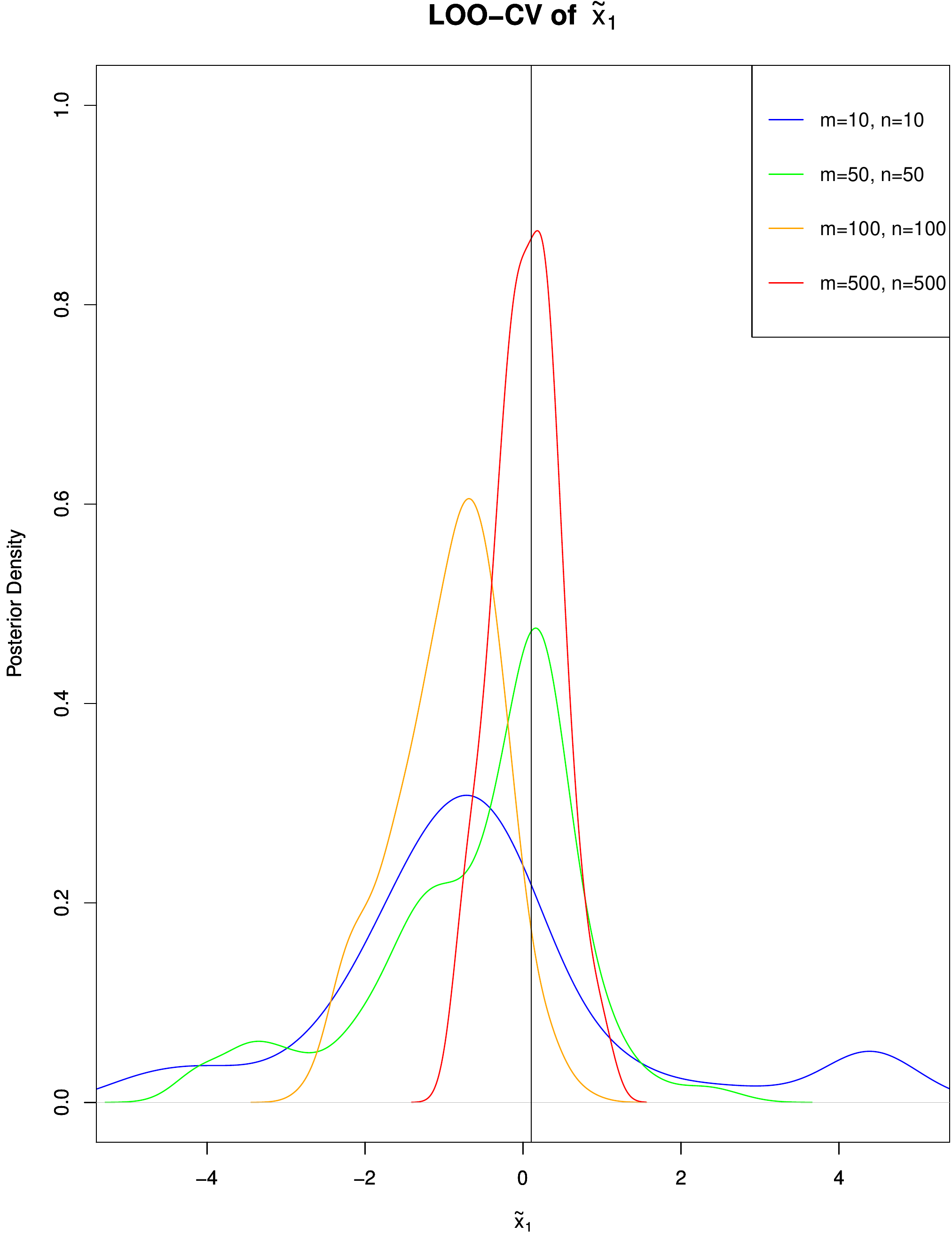}
\end{center}
\caption{Demonstration of posterior consistency in inverse nonparemetric Poisson regression. The vertical line denotes the true value.}
\label{fig:poisson_nonparametric}
\end{figure}

\section{Conclusion}
\label{sec:conclusion}
In this paper, we have proposed a prior for $\tilde x_i$ that seems to be natural for ensuring consistency of the LOO-CV posteriors, and hence
of the IRD approach. Crucially, we need $m$ observations corresponding to each $x_i$, and $m$ is taken to infinity for the asymptotic theory.
Note that for $m=1$, or for any finite $m$, consistency of the LOO-CV posterior of $\tilde x_i$ not achievable, even though consistency of the corresponding posterior of $\theta$
is attainable for any $m\geq 1$. This issue sets apart the problem of LOO-CV consistency from the usual parameter consistency. 

An interesting issue is that, for forward Bayesian problems, the posterior predictive distribution of the $i$-th response $y_i$ does not tend to point mass at $y_i$, 
even if the corresponding posterior of $\theta$ is consistent at $\theta_0$. The reason is that the distribution of $y_i$ given $\theta$ and $x_i$ is specified
as per the modeled likelihood, and does not admit any prior construction as in the inverse setup. Since the modeled response variable is always associated with positive
variability, even under the true model, the posterior predictive distribution of $y_i$ always has positive variance, and hence, can not be consistent at $y_i$.
From this perspective, even in forward problems, it perhaps makes sense to consider the IRD approach for model validation. 
Indeed, our simulation studies demonstrate the effectiveness of the IRD approach to model validation compared to the forward approach.

As a final remark, we mention that for our prior on $\tilde x_i$ we required independence among $\{y_{i1},\ldots,y_{im}\}$, for the strong law of large numbers
to hold for $\bar y_i$ and $s^2_i$. However, independence is not strictly necessary, as the ergodic theorem can often be utilized for ensuring limits in the strong sense.

\begin{appendix}

\section*{Appendix}

\section{Preliminaries for ensuring posterior consistency under general setup}
\label{sec:shalizi}

Following \ctn{Shalizi09} we consider a probability space $(\Omega,\mathcal F, P)$, 
and a sequence of random variables $y_1,y_2,\ldots$,   
taking values in some measurable space $(\Xi,\mathcal Y)$, whose
infinite-dimensional distribution is $P$. Let $\bY_n=\{y_1,\ldots,y_n\}$. The natural filtration of this process is
$\sigma(\bY_n)$, the smallest $\sigma$-field with respect to which $\bY_n$ is measurable. 

We denote the distributions of processes adapted to $\sigma(\bY_n)$ 
by $F_{\theta}$, where $\theta$ is associated with a measurable
space $(\Theta,\mathcal T)$, and is generally infinite-dimensional. 
For the sake of convenience, we assume, as in \ctn{Shalizi09}, that $P$
and all the $F_{\theta}$ are dominated by a common reference measure, with respective
densities $f_{\theta_0}$ and $f_{\theta}$. The usual assumptions that $P\in\Theta$ or even $P$ lies in the support 
of the prior on $\Theta$, are not required for Shalizi's result, rendering it very general indeed.

\subsection{Assumptions and theorems of Shalizi}
\label{subsec:assumptions_shalizi}

\begin{itemize}
\item[(S1)] Consider the following likelihood ratio:
\begin{equation*}
R_n(\theta)=\frac{f_{\theta}(\bY_n)}{f_{\theta_0}(\bY_n)}.
\end{equation*}
Assume that $R_n(\theta)$ is $\sigma(\bY_n)\times \mathcal T$-measurable for all $n>0$.
\end{itemize}

\begin{itemize}
\item[(S2)] For every $\theta\in\Theta$, the KL-divergence rate
\begin{equation*}
h(\theta)=\underset{n\rightarrow\infty}{\lim}~\frac{1}{n}E\left(\log\frac{f_{\theta_0}(\bY_n)}{f_{\theta}(\bY_n)}\right).
\end{equation*}
exists (possibly being infinite) and is $\mathcal T$-measurable.
\end{itemize}

\begin{itemize}
\item[(S3)] For each $\theta\in\Theta$, the generalized or relative asymptotic equipartition property holds, and so,
almost surely,
\begin{equation*}
\underset{n\rightarrow\infty}{\lim}~\frac{1}{n}\log R_n(\theta)=-h(\theta).
\end{equation*}
\end{itemize}

\begin{itemize}
\item[(S4)] 
Let $I=\left\{\theta:h(\theta)=\infty\right\}$. 
The prior $\pi$ satisfies $\pi(I)<1$.
\end{itemize}

\begin{itemize}
\item[(S5)] There exists a sequence of sets $\mathcal G_n\rightarrow\Theta$ as $n\rightarrow\infty$ 
such that: 
\begin{enumerate}
\item[(1)]
\begin{equation}
\pi\left(\mathcal G_n\right)\geq 1-\alpha\exp\left(-\beta n\right),~\mbox{for some}~\alpha>0,~\beta>2h(\Theta);
\label{eq:S5_1}
\end{equation}
\item[(2)]The convergence in (S3) is uniform in $\theta$ over $\mathcal G_n\setminus I$.
\item[(3)] $h\left(\mathcal G_n\right)\rightarrow h\left(\Theta\right)$, as $n\rightarrow\infty$.
\end{enumerate}
\end{itemize}
For each measurable $A\subseteq\Theta$, for every $\delta>0$, there exists a random natural number $\tau(A,\delta)$
such that
\begin{equation}
n^{-1}\log\int_{A}R_n(\theta)\pi(\theta)d\theta
\leq \delta+\underset{n\rightarrow\infty}{\lim\sup}~n^{-1}
\log\int_{A}R_n(\theta)\pi(\theta)d\theta,
\label{eq:limsup_2}
\end{equation}
for all $n>\tau(A,\delta)$, provided 
$\underset{n\rightarrow\infty}{\lim\sup}~n^{-1}\log\pi\left(\mathbb I_A R_n\right)<\infty$.
Regarding this, the following assumption has been made by Shalizi:
\begin{itemize}
\item[(S6)] The sets $\mathcal G_n$ of (S5) can be chosen such that for every $\delta>0$, the inequality
$n>\tau(\mathcal G_n,\delta)$ holds almost surely for all sufficiently large $n$.
\end{itemize}
\begin{itemize}
\item[(S7)] The sets $\mathcal G_n$ of (S5) and (S6) can be chosen such that for any set $A$ with $\pi(A)>0$, 
\begin{equation}
h\left(\mathcal G_n\cap A\right)\rightarrow h\left(A\right),
\label{eq:S7}
\end{equation}
as $n\rightarrow\infty$.
\end{itemize}

\end{appendix}

\bibliography{irmcmc}

\end{document}